\documentclass[12pt]{amsart}
\usepackage{amssymb}
\usepackage{mathrsfs}
\usepackage[bitstream-charter]{mathdesign}
\usepackage[colorlinks]{hyperref}
\usepackage{color}
\usepackage[all]{xy}
\input diagxy
\input xy
\xyoption{all}

\newtheorem{theorem}{Theorem}[section]

\newtheorem{conj}[theorem]{Conjecture}

\theoremstyle{definition}

\theoremstyle{remark}

\numberwithin{equation}{section}

\def \n{\noindent}

\def \mb{\mathbb}
\def \mc{\mathcal}

\def \fk{\mathfrak}

\def \scr{\mathscr}

\def \cplane{\mathbb{C}}              %complex number
\def \integer{\mathbb{Z}}             %integer number
\def \natu{\mb N}                     %natural number

\def \rto{\rightarrow}

\def \hr{\hookrightarrow}
\def \0{\infty}                       %infinity
\def \P{\mb P}                        %Projection

\def \qq{\quad}
\def \1{{\bf 1}}

\def \F{{\scr F}}
\def \M{\overline{\scr M}}
\def \bl{\Big\langle}
\def \br{\Big\rangle}
\def \bb{\Big |}
\def \aut{\text{Aut}}

\def \cN{\mc N}
\def \F{\scr F}
\def \simple{{\text{simple}}}
\def \vir{{\text{vir}}}

\def \x{{\vec x}}
\def \y{{\vec y}}
\def \bF{{\overline{\F}}}

\allowdisplaybreaks
\begin{document}

\title[Relative GW--invariants of projective completions]
{On relative Gromov--Witten invariants of
projective completions of vector bundles}

\author{Cheng-Yong Du}
\address{School of Mathematics, Sichuan Normal University,
Chengdu, China, 610068}
%\curraddr{Department of Mathematics and Statistics,
%Case Western Reserve University, Cleveland, Ohio 43403}

\email{cyd9966@hotmail.com}

%\author{Bohui Chen}
%\address{Department of Mathematics and Yangtze Center of Mathematics, Sichuan %University, 610065, Chengdu, %People's Republic of China.}
%\email{bohui@cs.wisc.edu.}

%\thanks{The author is supported by NSFC (No. 11501393).}

%    General info

\date{}%{January 1, 2001 and, in revised form, June 22, 2001.}

%\dedicatory{This paper is dedicated to my advisors.}


\begin{abstract}
It was proved by Fan--Lee %\cite{FL16}
and Fan %\cite{Fan17}
that the absolute Gromov--Witten invariants of two projective bundles $\P(V_i)\rto X$ are identified canonically when the total Chern classes $c(V_1)=c(V_2)$ for two bundles $V_1$ and $V_2$ over a smooth projective variety $X$. In this note we show that for the two projective completions $\P(V_i\oplus\mc O)$ of $V_i$ and their infinity divisors $\P(V_i)$, the relative Gromov--Witten invariants of $(\P(V_i\oplus\mc O),\P(V_i))$ are identified canonically when $c(V_1)=c(V_2)$.
%prove an analogue for relative Gromov--Witten invariants of the projective completions of vector bundles by combining results of Fan--Li \cite{FL16} and Fan \cite{Fan17}, and the topological view of Gromov--Witten theory of Maulik--Pandharipande \cite{MP06}. We prove this result in the symplectic category.
\end{abstract}

\maketitle

%\tableofcontents
\section{Statement of the main result} \label{sec 1}

Let $V_i\rto X, i=1,2$ be two rank $n$ vector bundles over a smooth projective variety $X$. Denote the total Chern classes of $V_i,i=1,2$ by
\[
c(V_i):=1+c_1(V_1)+\ldots+c_n(V_i).
\]
Suppose that $c(V_1)=c(V_2)$. Let $Y_i:=\P(V_i\oplus\mc O_X)$ be the projective completion (or projectification) of $V_i$. Denote the projection maps by $\pi_i:Y_i\rto X$. By Leray--Hirsch we get
\[
H^*(Y_i)=\frac{H^*(X)[y_i]}
{(y_i^{n+1}+y_i^nc_1(V_i)+\ldots+y_ic_n(V_i))},
\qq i=1,2,
\]
where $y_i=c_1(\mc O_{Y_i}(1))$ is the first Chern class of the dual bundle of the tautological line bundle $\mc O_{Y_i}(-1)$ over $Y_i$.

Because the Chern classes of $V_1$ and $V_2$ agree, we have the following isomorphism
\[
\F_Y: H^*(Y_1)\cong H^*(Y_2),
\]
satisfying $\F_Y(\pi_1^*(\sigma))=\pi_2^*(\sigma), {\F_Y(y_1)=y_2}$. As to their numerical curve classes, there is a unique isomorphism (see also \cite{FL16,Fan17})
\[
\Psi: N_1(Y_1)\cong N_1(Y_2)
\]
determined by the property
\[
\langle D,\beta\rangle_{Y_1}=\langle\F_Y(D),\Psi(\beta)\rangle_{Y_2},
\]
where $D\in H^2(Y_1;\integer)$ and $\langle\cdot,\cdot\rangle_{Y_i}$ are the Poincar\'e {pairings} of $Y_i$.

In fact, we have
\[
N_1(Y_i)=\iota_{i,*}(N_1(X))\oplus \integer[F_i],
\]
where for $i=1,2$, $F_i$ is the fiber class of $\pi:Y_i\rto X$, and $\iota_i:X\rto Y_i$ is the inclusion as the zero section $\P(0\oplus\mc O_X)$ in $Y_i$. Then
\[
\Psi(\iota_{1,*}(\beta)+dF_1)=\iota_{2,*}(\beta)+dF_2.
\]

Similarly for the projectivization $Z_i:=\P(V_i)$, which is the infinite divisor in $Y_i$ via the inclusion
$\iota_i:Z_i=\P(V_i\oplus 0)\hr Y_i$ for $i=1,2$, we have
\[
H^*(Z_i)=\frac{H^*(X)[z_i]}{(z_i^n+z_i^{n-1}c_1(V_i)+\ldots+c_n(V_i))},
\]
where
\[
z_i=c_1(\mc O_{Z_i}(1))=\iota^*_i(c_1(\mc O_{Y_i}(1)))=\iota^*_i(y_i).
\]
We also have
\[
\F_Z:H^*(Z_1)\cong H^*(Z_2)
\]
satisfying $\F_Z(\pi^*_1(\sigma))=\pi_2^*(\sigma)$ and $\F_Z(z_1)=z_2$. Since via the inclusions $\iota_i$ we could identify the fiber classes of $Y_i\rto X$ and $Z_i\rto X$, $\Psi$ restricts to $N_1(Z_1)\cong N_1(Z_2)$. Our main theorem is

\begin{theorem}\label{thm main}
Let $X$ be a smooth projective variety,
and $V_1, V_2$ be two rank $n$ vector bundles over $X$.
Let $Y_1, Y_2$, $\F_Y,\F_Z$ be the same as above.
Suppose $c(V_1)=c(V_2)$.
Then we have the following
equality between relative Gromov--Witten invariants
\begin{align}\label{eq main-thm}
\bl \varpi \bb \mu \br^{Y_1,Z_1}_{g,m,\beta,\vec\mu}=
\bl \F_Y(\varpi) \bb \F_Z(\mu)\br^{Y_2,Z_2}_{g,m,\Psi(\beta),\vec\mu}
\end{align}
with
\[
\varpi=(\tau_{k_1} \alpha_1 ,\ldots,\tau_{k_m}\alpha_m),
\]
\[
\F_Y(\varpi)=(\tau_{k_1} \F_Y(\alpha_1),\ldots,\tau_{k_m}\F_Y(\alpha_m))
\]
for any cohomology classes $\alpha_1,\ldots,\alpha_m\in H^*(Y_1)$
and any set of {nonnegative} numbers $k_1,\ldots,k_m$,
any curve class $\beta\in N_1(Y_1)$, any genus $g\in \natu$
and any $H^*(Z_1)$-weighted partition of $[Z_1]\cdot \beta$
\[
\mu=((\mu_1,\gamma_1),\ldots,(\mu_{\ell(\mu)},\gamma_{\ell(\mu)}))
\]
with $\gamma_i\in H^*(Z_1)$,
$\vec\mu=(\mu_1,\ldots,\mu_{\ell(\mu)})$, and
\[
\F_Z(\mu)=
((\mu_1,\F_Z(\gamma_1)),\ldots,(\mu_{\ell(\mu)},\F_Z(\gamma_{\ell(\mu)}))).
\]
\end{theorem}

We prove Theorem \ref{thm main} via  virtual localization for relative Gromov-Witten invariants \cite{GV05}, the topological view for Gromov--Witten theory of Maulik--Pandharipande \cite{MP06}, and the results of Fan--Lee \cite{FL16} and Fan \cite{Fan17} on the uniqueness of absolute Gromov--Witten theory of projective bundles.

In proving this theorem we need to consider the rubber invariants of the projective completions of the normal line bundles of $Z_i$ in $Y_i$ for $i=1,2$. The normal line bundles of $Z_i$ in $Y_i$ are $\mc O_{Z_i}(1)$. For simplicity we denote them by $L_i$. The projective completions of $L_i$ are
\[
W_i:=\P(L_i\oplus\mc O_{Z_i}),
\]
with projection $\tilde \pi_i:W_i\rto Z_i$. Their cohomologies are
\[
H^*(W_i)=\frac{H^*(Z_i)[w_i]}{(w_i^2+z_iw_i)},\qq \mbox{for}\qq i=1,2,
\]
with $w_i=c_1(\mc O_{W_i}(1))$. Then we also have an identification $\F_W:H^*(W_1)\cong H^*(W_2)$ satisfying $\F_W(\tilde \pi^*\gamma)=\tilde \pi^*(\F_Z(\gamma))$ for $\gamma\in H^*(Z_1)$ and $\F_W(w_1)=w_2$. Therefore one can see that $\F_W$ is determined by $\F_Z$.

In the following we denote all fibration projections from $Y_i,Z_i,W_i$ to $X$ by $\pi_i$, the fibration projections from $W_i$ to $Z_i$ by $\tilde\pi$. Then $\pi_i=\pi_i\circ\tilde\pi_i:W_i\rto X$. The fiber of $\tilde\pi:W_i\rto Z_i$ is $\P^1$, the fiber of $\pi:W_i\rto X$ is the one point blow-up of $\P^n$, i.e. the projective completion of $\mc O_{\P^{n-1}}(1)$. Denote the fiber class of $\tilde\pi$ by $\tilde F_i$, and the fiber class of $\pi:Z_i\rto X$ and $\pi:Y_i\rto X$ by $F_i$. We could extend $\Psi:N_1(Z_1)\cong N_1(Z_2)$ to
\[
\Psi:N_1(W_1)\cong N_1(W_2),\qq
\Psi(\tilde\iota_{1,*}(\beta)+d[\tilde F_1])
=\tilde\iota_{2,*}({\Psi}(\beta))+d[\tilde F_2],
\]
where $\tilde\iota_i:Z_i\rto W_i$ are the inclusions as zero sections for $i=1,2$.

In the following, when we consider a single vector bundle $V\rto X$, all notations above apply.

%\section{Some basic facts about projective bundles}

%We first review some fact about projective bundles from \cite{Fan17}.

%\begin{lemma}[\cite{Fan17}]\label{lem canonical-divisor}
%We have $(K_{Y_1},\beta)=(K_{Y_2},\Psi(\beta))$.
%\end{lemma}

%Like the absolute case,
%in proving the theorem,
%we are free to twist both $V_1$ and $V_2$ by a fixed line bundle
%simultaneously.
%To be more precise, for any line bundle $L$ on $X$,
%we still have
%\[
%c(V_1 \otimes L\inv)=c(V_2\otimes L\inv)
%\]
%and
%\[
%\big(P((V_i\oplus\mc O_X)\otimes L\inv), \P(V_i\otimes L\inv)\big)
%\cong
%\big(P(V_i\oplus\mc O_X),\P(V_i)\big).
%\]

%We next consider only a single vector bundle $\pi: V\rto X$.
%Denote its projectification by $\pi:Y=\P(V\oplus\mc O_X)\rto X$,
%and the infinite divisor, i.e. the projectivization by $\pi:Z:=\P(V)\rto X$.
%We have
%\begin{lemma}[\cite{Fan17}] \label{lem ampleness-the-bundle}
%For any $g\in\natu$, there is a %sufficiently ample
%line bundle
%$L_g\in\mbox{Pic}(X)$ such that
%for any $\beta\in\mbox{NE}(Y)$ with $\pi_*(\beta)\neq 0$,
%we have the intersection pairing
%\[
%(\beta,\mc O_Y(1)+\pi^*L_g)> \max\{g-1,0\},
%\]
%(where the notation $\mbox{NE}(\cdots)$ stands for the Mori cone,
%i.e. the cone of effective classes).
%\end{lemma}

%In the rest part we always fix such an $L_g$.

\section{Relative Gromov--Witten invariants of fiber class}\label{sec fiber-class-inv}

We first consider fiber class relative Gromov--Witten invariants, i.e. those invariants with homology class being a fiber class of $\pi:Y_i\rto X$.

We consider a single vector bundle $V\rto X$. Let $F$ denote the fiber class of $\pi:Y=\P(V\oplus\mc O_X)\rto X$, $Z$ be the infinite divisor $\P(V)$ of $Y$. Following \cite{MP06}, we see that fiber class relative Gromov--Witten invariants of $(Y,Z)$ reduces to relative Gromov--Witten invariants of $(\P^n,\P^{n-1})$ as follow. Consider a fiber class relative moduli space of $(Y,Z)$
\[
\M_{g,m,dF,\vec\mu}(Y,Z).
\]
There is a fibration structure
\[
\M_{g,m,dF,\vec\mu}(\P^n,\P^{n-1})\hr
\M_{g,m,dF,\vec\mu}(Y,Z)\stackrel{\pi}{\rto} X
\]
whose fiber is the relative moduli space of $(\P^n,\P^{n-1})$. This fibration induces a $\pi$-relative virtual fundamental class $[\M_{g,m,dF,\vec\mu}(Y,Z)]^{\text{vir}_\pi}$ such that
\[
[\M_{g,m,dF,\vec\mu}(Y,Z)]^{\text{vir}}= c(TX\boxtimes \mb E) \cap[\M_{g,m,dF,\vec\mu}(Y,Z)]^{\text{vir}_\pi},
\]
where $\mathbb E$ is the Hodge bundle. We then have
\[\begin{split}
&[\M_{g,m,dF,\vec\mu}(Y,Z)]^{\text{vir}}\\
=&
\sum_q h_q(c_1(\mathbb E),c_2(\mathbb E),\ldots)
t_q (c_1(TX),c_2(TX ),\ldots)
\cap[\M_{g,m,dF,\vec\mu}(Y,Z)]^{\text{vir}_\pi}.
\end{split}\]
Therefore for
\[
\alpha_i=\pi^*(\sigma_i)y^{a_i},\qq
\mu=(\ldots,(\mu_j,\pi^*(\gamma_j) z^{b_j}),\ldots)
\]
with $\sigma_i,\gamma_j\in H^*(X)$, we have
\begin{align*}
\begin{split}
&\bl \tau_{k_1}\alpha_1,\ldots,\tau_{k_m}\alpha_m\bb
\mu\br^{Y,Z}_{g,m,dF,\vec\mu}\\
=&\frac{1}{|\aut(\mu)|}\sum_q\int_X t_q\prod_i\sigma_i\prod_j\gamma_j
\Big(\int_{[\M_{g,m,dF,\vec\mu}(Y,Z)]^{\text{vir}_\pi}}
h_q\prod_i({\psi^{k_i}_i}y^{a_i})\prod_j z^{b_j}\Big)
\end{split}
\end{align*}

\n The integration
\[
\int_{[\M_{g,m,dF,\vec\mu}(Y,Z)]^{\text{vir}_\pi}}
h_q\prod_i({\psi^{k_i}_i}y^{a_i})\prod_j z^{b_j}
\]
are obtained from the Hodge integrals of descendent relative Gromov-Witten invariants of $(\P^n,\P^{n-1})$, which only depends on $\vec\mu$, $(k_1,\ldots,k_m)$, $(a_1,\ldots,a_m)$ and $(b_1,\ldots,b_{\ell(\mu)})$. The transformations $\F_Y$ and $\Psi$ do not change these datum. Therefore Theorem \ref{thm main} holds for  fiber class relative Gromov--Witten invariants of $(Y_i,Z_i)$.

\section{Relative Gromov--Witten invariants with general homology classes}

In this section we consider Gromov--Witten invariants of $(Y_i,Z_i)$ with general homology classes.

\subsection{Relative virtual localization}

We first also consider a single vector bundle $V\rto X$. Then we have $Y=\P(V\oplus\mc O_X)$, $Z=\P(V)$, $W=\P(L\oplus\mc O_Z)$ with $L=\mc O_Z(1)$ being the normal line bundle of $Z$ in $Y$ as in Section \ref{sec 1}. $W$ has {the} zero section $Z_0$ and {the} infinite section $Z_\0$, which are both isomorphic to $Z$. We also use $F$ to denote the fiber class of $\pi:Y\rto X$ and $\pi:Z\rto X$, and $\tilde F$ to denote the fiber class of $\tilde\pi:W\rto Z$.

We consider relative Gromov--Witten invariants with non-fiber homological classes of $(Y,Z)$. After \cite{MP06}, we use virtual localization for relative Gromov--Witten invariants, which was worked out explicitly by Graber and Vakil \cite{GV05}, to reduce relative Gromov--Witten invariants of $(Y,Z)$ into twisted Gromov--Witten invariants of $X$ with twisting coming from $V$ and {\em rubber invariants} of $(W,Z_0\sqcup Z_\0)$. Then via the rubber calculus in \cite[Section 1.5]{MP06}, the rubber invariants are related to fiber class relative Gromov--Witten invariants and {\em distinguished type II invariants} of $(W,Z_0\sqcup Z_\0)$. At last distinguished type II invariants of $(W,Z_0\sqcup Z_\0)$, are determined by an induction algorithm in \cite{MP06} with fiber class relative Gromov--Witten invariants as the initial data.

Take a self dual basis
\[
\Sigma_\star:=\{\sigma_i\}_{1\leq i\leq N}
\]
of $H^*(X)$.

Consider the invariant of $(Y,Z)$
\begin{align}\label{eq rel-inv-considered}
\bl\varpi\bb\mu\br^{Y,Z}_{g,m,\beta,\vec\mu}
\end{align}
with
\[
\varpi=(\tau_{k_1}\alpha_1,\ldots,\tau_{k_m}\alpha_m)
\]
and
\[
\mu=((\mu_1,\gamma_1),\ldots,
(\mu_{\ell(\mu)},\gamma_{\ell(\mu)})).
\]
Denote the corresponding moduli space by $\M_\Gamma$ with $\Gamma$ denoting the topological data
\[
\Gamma:=(g,m,\beta,\vec\mu).
\]
Then $\beta=\pi_*(\beta)+dF$ with $d\geq 0$, $\pi_*(\beta)\in N_1(X)$, and $|\vec\mu|=\sum_i\mu_i=d$.

There is a $\cplane^*$-action on $Y=\P(V\oplus\mc O_X)$ via dilation on the fiber of $V$. This action has two fixed loci:
\begin{itemize}
\item one is $X$ embedded into $Y$ as the zero section,      whose normal bundle is $V\rto X$,
\item the other one is the infinite divisor {$Z=\P(V\oplus 0)=\P(V)$}, whose normal bundle is $\mc O_Z(1)$, denoted by $L$.
\end{itemize}
%The cohomology classes of $H^*(Y)$ and $H^*(Z)$ have natural equivariant liftings (cf. \cite{Fan17}).

Stable maps in $\M_\Gamma$ consists of two types, those mapped to the rigid target $(Y,Z)$, and those mapped to a non-rigid target $(Y[l],Z)$. Here $Y[l]$ is obtained as follow. Take $l$ copies of $W$. Denote the zero section and infinity section of the $i$-th copy of $W$ by $Z_{0,i}$ and $Z_{\0,i}, 1\leq i\leq l$. We first glue the $l$ copies of $W$ together to get $W[l]$ via identifying $Z_{0,i}$ with $Z_{\0,i+1}$ for $1\leq i\leq l-1$. Then $Y[l]$ is obtained by gluing $Y$ with $W[l]$ by identifying $Z$ with $Z_{\0,1}\in W[l]$. We sometimes denote the $Z$ in $Y$ by $Z_{0,0}$. Then
\[
\mbox{Sing}(Y[l])=
Z\sqcup\coprod_{i=1}^{l-1}Z_{0,i}
=\coprod_{i=1}^l Z_{\0,i}.
\]
The $Y$ in $Y[l]$ is called the root, and the rest $W[l]$ is called rubber.

Therefore there are two types of fixed locus of the induced $\cplane^*$-action on $\M_\Gamma$. A component of the fixed loci consisting of general stable maps with target $Y$ is call a {\em simple} fixed locus. Otherwise, it is called a {\em composite} fixed locus. We denote the simple fixed locus by $\M_\Gamma^{\simple}$. Denote the virtual normal bundle of $\M_\Gamma^\simple$ in $\M_\Gamma$ by $\cN_\Gamma$.

Any element of a composite fixed locus is of the form $f:C'\cup C''\rto Y[l]~(l\geq 1)$, such that the restrictions $f': C'\rto Y$ and $f'': C''\rto W[l]$ agree over the nodes $\{N_1,\cdots,N_\ell\}=C'\cap C''$. Suppose the contact order of $f'$ at $N_i$, i.e. at $Z_{\0,1}(=Y\cap W[l])$ in $Y[l]$, is $\eta_i$ for $1\leq i\leq \ell$. Let $\Gamma'$ be the topological data corresponding to $f'$ and $\Gamma''$ the topological data corresponding to $f''$. (Here $\Gamma''$ denote the genus, absolute markings, degree and contact orders of relative markings relative to both $Z_{0,l}$ in $W[l]$ and contact orders at those nodes $N_i, 1\leq i \leq \ell$.) Any two of $\{\Gamma,\Gamma',\Gamma''\}$ determine the third, and $\Gamma'$ gives us a simple fixed locus $\M_{\Gamma'}^\simple$, and $f'\in\M_{\Gamma'}^\simple$.
%with $\eta_i=\frac{m_i}{r_i}$, where $\f(N_i)\in \D_{(g_i)}$ and $\text{ord}(g_i)=r_i$. This is part of the data of both $\Gamma'$ and $\Gamma''$, i.e. the contact order of $\f''$ at $N_i$ is $\eta_i$ too.

Denote by $\M_{\Gamma''}^\sim$ the moduli space of relative stable maps to the rubber (see Graber--Vakil \cite{GV05}, Maulik--Pandharipande \cite{MP06}). Then the fixed locus $\bF_{\Gamma',\Gamma''}$ corresponding to a given $\Gamma'$ and $\Gamma''$ is canonically isomorphic to the quotient of the moduli space
\[
\M_{\Gamma',\Gamma''}
:=\M_{\Gamma'}^\simple\times_{(Z)^\ell}\M_{\Gamma''}^\sim
\]
by the finite group $\aut(\vec\eta)$, which consists of permutations of $\{1,\ldots,\ell\}$ preserving
\[
\vec\eta={(\eta_1,\ldots, \eta_\ell)}.
\]
Denote the quotient map by $gl$:
\[
gl: \M_{\Gamma',\Gamma''}\rto \F_{\Gamma',\Gamma''}.
\]
Set
\[
[\M_{\Gamma',\Gamma''}]^\vir
:=\Delta^!([\M_{\Gamma'}^\simple]^\vir\times
[\M_{\Gamma',\Gamma''}^\sim]^\vir)
\]
where $\Delta: Z^\ell\rto Z^\ell\times Z^\ell$ is the diagonal map. Then we have \cite{GV05}
\[
[\bF_{\Gamma',\Gamma''}]^\vir=
\frac{1}{|\aut(\vec\eta)|}gl_\ast[\M^{\Gamma',\Gamma''}]^\vir.
\]

The virtual normal bundle of the composite locus $\bF_{\Gamma',\Gamma''}$ consists of two parts. The first part is the virtual normal bundle $\cN_{\Gamma'}$ of $\M_{\Gamma'}^\simple$ in $\M_{\Gamma'}$. The second part is a line bundle $\mc L$ corresponding to the deformation (i.e. smoothing) of the singularity $D_{\0,1}(=Y\cap W[l])$ in $W[l]$. The fiber of this line bundle over a point in the fixed locus is canonically isomorphic to $H^0(Z,N_{Z|Y}\otimes N_{Z_{\0,1}|W[l]})$. The line bundle $N_{Z|Y}\otimes N_{Z_{\0,1}|W[l]}=L\otimes L^*=\mc O_Z$ is trivial {over} $Z$, so its space of global sections is one-dimensional, and we can canonically identify this space of sections with the fiber of the line bundle at a generic point {``$pt$''} of $Z$. Thus we can write the bundle $\mc L$ as a tensor product of bundles pulled back from the two factors separately. The one coming from $\M_{\Gamma'}$ is trivial, since it is globally identified with $H^0(pt,N_{Z|Y}\big|_{pt})$, but it has a nontrivial torus action; we denote this weight by $t$. The line bundle coming from $\M_{\Gamma''}^\sim$ is a nontrivial line bundle, which has fiber $H^0(pt, N_{Z_{\0,1}|W[l]}\big|_{pt})$, but has trivial torus action. We denote its first Chern class by $\Psi_\0$.

The relative virtual localization for relative Gromov--Witten invariants (cf. \cite[Theorem 3.6]{GV05}) is
\begin{align}\label{eq vir-local-formula}
[\M_\Gamma]^\vir=\frac{[\M_\Gamma^\simple]^\vir}{e(\cN_\Gamma)}+
\sum_{\substack{\M_{\Gamma',\Gamma''} \text{ composite}}}
\frac{(\prod_i \eta_i) gl_*[\M_{\Gamma',\Gamma''}]^\vir}
{|\aut(\vec\eta)|e(\cN_{\Gamma'})(t+\Psi_\0)}.
\end{align}

We next describe explicitly the fixed loci of $\M_\Gamma$. We first consider the simple fixed locus $\M_\Gamma^\simple$. An equivalent class of maps $[f: (C,\x,\y)\rto (Y,Z)]$ in the simple fixed locus must have the following form:
\begin{itemize}
\item $(C,\x,\y)$ with absolute markings $\x=(x_1,\ldots,x_m)$, and relative markings $\y=(y_1,\ldots,y_{\ell(\mu)})$ is of the form
       \[
       C=C_0\cup C_1\cup\ldots\cup C_{\ell(\mu)}
       \]
     with $C_0\cap C_i=\{N_i\}$ being a nodal point, called {\em distinguished node}, and $C_i\cap C_j=\varnothing$ for $1\leq i<j\leq \ell(\mu)$. Moreover $x_i\in C_0, 1\leq i\leq m$, and $y_j\in C_j,1\leq j\leq \ell(\mu)$.

\item $C_0$ is a genus $g$ pre-stable curves with $m$ marked points $\x$ and $\ell(\mu)$ marked points $\vec n=(N_1,\cdots,N_{\ell(\mu)})$ corresponding to the $\ell(\mu)$ distinguished nodes.

\item For $1\leq i\leq \ell(\mu)$, each $C_i$ is a Riemann sphere with a marking $y_i$ and a marking $N_i$ corresponding to the $i$-th distinguished node.

\item $f: (C_0,\x\sqcup\vec n)\rto X$ is a genus $g$ degree     $\pi_*(\beta)$ stable maps to $X$, and whose equivalent class belongs to the moduli space $\M_{g,m+\ell(\mu),\pi_*(\beta)}(X)$.

\item For $1\leq i\leq \ell(\mu)$, $f: (C_i,y_i,N_i)\rto Y$ is a total ramified covering of a line in the fiber of $\pi:Y\rto X$ that connects a point in $Z$ and a point in $X$, the degree is determined by the contact order at $Z$. Hence $[f: (C_i,y_i,N_i)\rto Y]$ is in the simple fixed loci of the moduli space
      \[
      \M_{0,1,\mu_i[F],(\mu_i)}(Y,Z),
      \]
    the moduli space of fiber class $\mu_i[F]$ stable maps from Riemannian spheres with exactly one absolute marking mapped to $Y$ and one relative marking mapped to $Z$ with contact order $\mu_i$. For simplicity, we denote this simple fixed locus by $\M_{\mu_i}^\simple$. Denote the disconnected union of $\M_{0,1,\mu_i[F], (\mu_i)}(Y,Z)$ by $\M^\bullet_\mu$, and the disconnected union of simple fixed locus {$\M_{\mu_i}^\simple$} by $\M^{\bullet,\simple}_\mu$.
\end{itemize}

\n Therefore, the simple fixed locus $\M_\Gamma^\simple$ is obtained by gluing stable maps in $\M_{g,m+\ell(\mu),\pi_*(\beta)}(X)$, and stable maps in $\M^{\bullet,\simple}_\mu$ along the absolute marked points of $C_0$ and $C_i$ corresponding to those distinguished nodes $\vec n$. Moreover
\[
gl: \M_\mu^{\simple}\times_{X^{\ell(\mu)}}
\M_{g,m+\ell(\mu),\pi_*\beta}(X)\rto \M_\Gamma^\simple
\]
is a degree $|\aut(\vec\mu)|$ cover. The fiber product is taken with respect to the evaluation maps at the marked points out of the $\ell(\mu)$ distinguished nodes {$\vec n$} of $C$.

The tangent space $T^1_{(C,\x,\y,f)}$ and the obstruction space $T^2_{(C,\x,\y,f)}$ at a moduli point $[(C,\x,\y,f)]\in\M_\Gamma^\simple$ fit in the following long exact sequence of $\cplane^*$-representations:
\begin{align*}
0&\rto \aut(C,\x,\y) \rto \text{Def}(f) \rto T^1_{(C,\x,\y,f)}\rto\\
&\rto \text{Def}(C,\x,\y) \rto \text{Obs}(f) \rto T^2_{(C,\x,\y,f)}\rto 0,
\end{align*}
where
\begin{itemize}
\item $\aut(C,\x,\y)=\text{Ext}^0(\Omega_C(\sum_ix_i+\sum_jy_j) \mc O_C)$ is the space of infinitesimal automorphism of the domain $(C,\x,\y)$. We have
       \[
       \aut(C,\x,\y)=\aut(C_0,\x,\vec n)\oplus
       \bigoplus_{i=1}^{\ell(\mu)}\aut(C_i,y_i,N_i),
       \]

\item $\text{Def}(C,\x,\y)=\text{Ext}^1(\Omega_C(-\sum_i x_i-\sum_jy_j),\mc O_C)$ is the space of infinitesimal deformation of the domain $(C,\x,\y)$. We have a short exact sequence of $\cplane^*$-representations:
       \[
       0 \rto \text{Def}(C_0,\x,\vec n)
         \rto \text{Def}(C,\x,\y)
         \rto \bigoplus_{i=1}^{\ell(\mu)}T_{N_i}C_0
              \otimes T_{N_i}C_i
         \rto 0,
       \]
\item $\text{Def}(f)=H^0(C,f^*(TY(-Z)))$ is the space of infinitesimal deformation of the map $f$, and

\item $\text{Obs}(f)=H^1(C,f^*(TY(-Z)))$ is the space of obstruction to deforming $f$.
\end{itemize}
Let
\[
B_1=\aut(C,\x,\y),\, B_2=\text{Def}(f),\,
B_4=\text{Def}(C,\x,\y),\, B_5= \text{Obs}(f)
\]
and let $B^f_i$ and $B^m_i$ be the fixed and moving parts of $B_i$ with respect to the induced $\cplane^\ast$-action. Then
\[
\frac{1}{e_{\cplane^*}(\cN_\Gamma)}
=
\frac{e_{\cplane^*}(B^m_5) e_{\cplane^*}(B^m_1)}
{e_{\cplane^*}(B^m_2)e_{\cplane^*}(B^m_4)}.
\]

On the other hand we have the following exact sequence
\begin{align}\label{eq exact-sequence-O}
0\rto \mc O_C\rto\bigoplus_{0\leq i\leq \ell(\mu)} \mc O_{C_i} \rto
\bigoplus_{1\leq i\leq \ell(\mu)} \mc O_{C_{\fk F_i}}
\rto 0.
\end{align}
Here $\fk F_i$ is the flag corresponding to the node $N_i$ and $C_0$.

Denote by $E=f^*(TY(-{Z}))$. Then the exact sequence \eqref{eq exact-sequence-O} gives us the following exact sequence
\begin{align*}
0&\rto H^0(C,E)\rto \bigoplus_{0\leq i\leq \ell(\mu)} H^0(C_i,E)
\rto \bigoplus_{1\leq i\leq \ell(\mu)} E_{N_i}\\
&\rto H^1(C,E)\rto \bigoplus_{0\leq i\leq \ell(\mu)} H^1(C_i,E) \rto0.
\end{align*}

\n Then the normal bundle $\cN_\Gamma$ of $\M_\Gamma^\simple$ in $\M_\Gamma$ consists of
\begin{itemize}
\item the normal bundle $\cN_\mu^\bullet=\coprod_i\cN_{\mu_i}$ of $\M_\mu^{\bullet,\simple}$ in $\M_\mu^\bullet$;

\item the contribution from deforming maps along $X$: $H^0(C_0,f^*V)- H^1(C_0,f^*V)$;

\item the contribution from deforming the distinguished nodes: $T_{N_i}C_0\otimes T_{N_i}C_i-V|_{f(N_i)}$, $1\leq i\leq\ell(\mu)$.
\end{itemize}
Denote the last two contribution by $\Theta_\Gamma$, it contains psi-class out of $T_{N_i}C_0$ and Chern class of $V$. Actually, the second item contributes the twisting coming from $V$ for twisted invariants of $X$.

Then we get
\begin{align*}
\int_{[\M_\Gamma^\simple]^\vir}
\frac{ev_\x^*\varpi\cup ev_\y^*\mu}{e_{\cplane^*}(\cN^\vir_\Gamma)}
&=\frac{1}{|\aut(\vec\mu)|}\cdot \\
\sum_{
\substack{
\sigma=(\sigma_{j_1},\ldots,\sigma_{j_{\ell(\mu)}})
\text{ in the}
\\
\text{chosen basis } \Sigma_\star \text{ of } H^*(X)}
}
\Big(
\int_{[\M_\mu^{\bullet,\simple}]^\vir} &
\frac{ev_\y^*\mu\cup ev_{\vec n}^*{\check\sigma}}
{e_{\cplane^*}({\cN_\mu^\bullet})}
\cdot
\int_{[\M_{g,m+\ell(\mu),\pi_*(\beta)}(X)]^{\vir}}
\frac{ev_\x^*\varpi\cup ev_{\vec n}^*{\sigma}}{\Theta_\Gamma}
\Big)\end{align*}
where the sum is taken over all possible $\ell(\mu)$-tuples of the chosen basis $\Sigma_\star$ of $H^*(X)$, and
\[
\int_{[\M_\mu^{\bullet,\simple}]^\vir} \frac{ev_\y^*\mu\cup
ev_{\vec n}^*{\check\sigma}}
{e_{\cplane^*}({\cN_\mu^\bullet})}
=\prod_{i=1}^{\ell(\mu)}
\int_{[\M_{\mu_i}]^\simple}
\frac{ev_{y_i}^*(\gamma_i)\cup ev_{N_i}^*({\check\sigma_{j_i}})}
{e_{\cplane^*}(\cN_{\mu_i})}
\]
is a product of fiber class (1+1)-point relative Gromov--Witten invariants of $(Y,Z)$. Such invariants were computed by Hu--Li--Ruan \cite[Section 7]{HLR08} and depend only on $\vec\mu$.

The invariant
\[
\int_{[\M_{g,\pi_*\beta,m+\ell(\mu)}(X)]^{\vir}}
\frac{ev_\x^*\varpi\cup ev_{\vec n}^*{\sigma}}{\Theta_\Gamma}
\]
is a Hodge integral over the twisted Gromov--Witten invariant of $X$ with twisting coming from the bundle $V$, which by the Quantum Riemann--Roch of Coates--Givental \cite{CG07}, is determined by the Gromov--Witten theory of $X$ and the total Chern class of $V\rto X$.

We next consider composite fixed locus. Recall that a composite fixed locus is of the form
\[
\bF_{\Gamma',\Gamma''}
=gl(\M_{\Gamma'}^\simple\times_{Z^{\ell(\vec\eta)}}
 \M_{\Gamma''}^\sim).
\]
with contribution being
\[
\frac{\prod_i \eta_i}{\aut(\vec\eta)}\cdot
\frac{gl_*\Delta^!([\M_{\Gamma'}^\simple\times
\M_{\Gamma''}^\sim]^{\vir})}{e(\cN_{\Gamma'})(t+\Psi_\0)}.
\]
By the analysis for simple fixed locus, the contributions from $\M_{\Gamma'}^\simple$ are determined by Gromov--Witten invariants of $X$ and total Chern class of the bundle $V$. Therefore, the contributions of $\bF_{\Gamma',\Gamma''}$ reduce to rubber invariants corresponding to $\M_{\Gamma''}^\sim$, which are rubber invariants with $\Psi^k_\0$-integrals of $(W,Z_0\sqcup Z_\0)$. Let $\beta'$ and $\beta''$ denote the homology classes of $\Gamma'$ and $\Gamma''$ respectively, then
\[\begin{split}
\beta  =\pi_*(\beta)+d[F],\qq
\beta' =\pi_*(\beta')+d'[F],\qq
\beta''=\tilde\pi_*(\beta'')+d[\tilde F]
\end{split}\]
with $\beta=\beta'+\tilde\pi_*(\beta'')$, hence $\tilde\pi_*(\beta'')=\pi_*(\beta)-\pi_*(\beta')+(d-d')[F]$.

We denote the rubber invariants of $(W,Z_0\sqcup Z_\0)$ by
\begin{align}\label{eq rubber-invariants}
\bl\mu\bb\varpi\cdot
\Psi_\0^k\bb\nu\br^{\sim}_{g,\beta}.
\end{align}
Note that here the $\varpi$ may be just a part of the initial one in \eqref{eq rel-inv-considered}, $g$ may less than the $g$ in \eqref{eq rel-inv-considered}, and $\beta$ is just the previous $\beta''$. There maybe disconnected rubber invariants in the contributions of $\bF_{\Gamma',\Gamma''}$.

\subsection{Rubber invariants}\label{sec ruber-invariant}

Via a boundary strata analysis and rigidification in \cite[Section 1.5]{MP06}, rubber invariants of $(W,Z_0\sqcup Z_\0)$ in \eqref{eq rubber-invariants} are determined by fiber class relative Gromov--Witten invariants of $(W,Z_0\sqcup Z_\0)$ and {\em distinguished type II invariants} of $(W,Z_0\sqcup Z_\0)$, i.e. invariants of the form:
\[
\bl \mu\bb\tau_0([Z_0]\cdot\delta)\cdot\varpi\bb\nu\br_{g,\beta}
\]
with $\delta\in H^{>0}(Z)$, where we omit the superscript $(W,Z_0\sqcup Z_\0)$. Denote this invariant by $R$. There is a partial order $\stackrel{\circ}{<}$ (cf. \cite[Section 1.3]{MP06}) over all distinguished type II invariants. Distinguished type II invariants are determined by induction via three relations in \cite[Section 1]{MP06}. The three relations express the invariant $R$ as a summation of products of twisted Gromov--Witten invariants of $Z$ with twisting coming from $L^\ast$, and distinguished type II invariants $R'$ that are lower than $R$ with respect to the partial order $\stackrel{\circ}{<}$. We have (cf. \cite{MP06})
\begin{eqnarray}\label{eq induction-type-II}
&&
C_R\bl \mu\bb\tau_0([Z_0]\cdot\delta)\cdot\varpi\bb\nu\br_{g,\beta}\\
&=&\sum_{\substack{||\varpi'||\leq ||\varpi||,
                  m\geq 0\\
                  \deg\mu'\geq \deg\mu\\
                  \deg\nu'\geq \deg\nu+1}}
                  C_{\mu',\varpi',\nu'}\cdot \bl\mu'\bb \tau_0([Z_0]\cdot
                  H\cdot c_1(L)^m)\cdot\varpi' \bb \nu' \br_{g,\beta}\nonumber\\
&+&\sum_{\substack{||\varpi'||\leq ||\varpi||,
                  m\geq 0\\
                  \deg\mu'\geq \deg\mu+1\\
                  \deg\nu'\geq \deg\nu}}
                  C_{\mu',\varpi',\nu'}\bl\mu'\bb \tau_0([Z_0]\cdot H\cdot
                  c_1(L)^m)\cdot\varpi' \bb \nu' \br_{g,\beta}\nonumber\\
&-&\sum_{\substack{||\varpi'||\leq ||\varpi||,
                  m\geq 0\\
                  \deg\mu'\geq \deg\mu\\
                  \deg\nu'\geq \deg\nu}}
                  C_{\mu',\varpi',\nu'}\bl\mu'\bb \tau_0([Z_0]\cdot \delta
                  \cdot c_1(L)^{m+1})\cdot\varpi' \bb \nu' \br_{g,\beta}\nonumber\\
&-&\sum_{\substack{||\varpi''||\leq ||\varpi||,
                  m\geq 0\\
                  \deg\mu''\geq \deg\mu+1\\
                  \deg\nu''\geq \deg\nu}}
                  C_{\mu'',\varpi'',\nu''}\bl\mu''\bb \tau_0([Z_0]\cdot
                  H\cdot c_1(L)^m)\cdot\varpi'' \bb \nu''\br_{g,\beta}\nonumber\\
&-&\sum_{\substack{R'_{g,\beta}:\text{ distinguished type II},\,
R'\stackrel{\circ}{<} R}}C_{R,R'} R' +\cdots, \nonumber
\end{eqnarray}
where

(1) $C_R$ is a constant depending on $\vec\nu$;

(2) all $C_{\ast,\ast}$ and $C_{\ast,\ast,\ast}$ are fiber invariants of $(W,Z_0)$ and $(W,Z_0\sqcup Z_\0)$;

(3) $H$ is a degree {2} class in $H^2(Z)$ such that $H\cdot \tilde\pi_*({\beta})\neq 0$;

(4) $\varpi'$ and $\omega''$ are parts of $\varpi$;

(5) $\mu',\nu',\mu''$ and $\nu''$ are new weighted partitions of $\beta\cdot[Z_0]$,

(6) $R'$ denote those distinguished type II invariants of $(W,Z_0\sqcup Z_\0)$ that have genus $g$ and homology class $\beta$ but are lower than $R$ with respect to the partial order $\stackrel{\circ}{<}$.

(7) We next focus on ``$\cdots$''. In the original relations
(cf. \cite[Relation 1, Relation 2 and Relation 2']{MP06}), ``$\cdots$'' stands for distinguished type II invariants of $(W,Z_0\sqcup Z_\0)$ and relative Gromov--Witten invariants of $(W,Z_0)$ with homology class $\beta'$ such that $\beta-\beta'$ is effective or $\beta=\beta'$ but genus $g'<g$. However, by virtual localization the latter ones are determined by Hodge integrals in twisted Gromov--Witten invariants of $Z$ with twisting coming from $\mc O_Z(-1)=L^\ast$, the normal bundle of $Z_\0$ in $W$, and rubber invariants of $(W,Z_0\sqcup Z_\0)$ with $\beta-\beta'$ is effective or $\beta=\beta'$ but genus $g'<g$. The latter one are determined by distinguished type II invariants of $(W,Z_0\sqcup Z_\0)$ that are lower than $R$. Therefore, in the formula above we use ``$\cdots$'' to stand for combination of Hodge integrals in twisted Gromov--Witten invariants of $Z$ with twisting coming from $L^\ast$,and distinguished type II invariants of $(W,Z_0\sqcup Z_\0)$ with homology class $\beta'$ such that $\beta-\beta'$ is effective or $\beta=\beta'$ but genus $g'<g$.

\section{Proof of Theorem \ref{thm main}}

Now consider $V_i\rto X$, $i=1,2$ that satisfy the conditions in Theorem \ref{thm main}. First of all, the fiber class relative Gromov--Witten invariants of $(Y_i,Z_i)$ are identified via \eqref{eq main-thm} by the analysis in Section \ref{sec fiber-class-inv}.

We next consider non fiber class relative Gromov--Witten invariants. By virtual localization, the contributions from simple fixed loci for the invariants in \eqref{eq main-thm} of both $(Y_i,Z_i)$ are identified via $\F_Y,\F_Z$ and $\Psi$. We next show that we could identify the contributions from composite fixed loci. We only need to show that rubber invariants of $(W_i,Z_{i,0}\sqcup Z_{i,\0})$, and fiber class relative Gromov--Witten invariants of $(W_i,Z_{i,0}\sqcup Z_{i,\0})$ and $(W_i,Z_{i,0})$ are identified via $\F_Z$ and $\F_Z(z_1)=z_2$ by the analysis in Section \ref{sec ruber-invariant}.

All fiber class relative Gromov--Witten invariants of $(W_i, Z_{i,0})$ and $(W_i, Z_{i,0}\sqcup Z_{i,\0})$ are identified similar as the analysis in Section \ref{sec fiber-class-inv} or \cite[Section 1.2]{MP06} as follow. Note that we identified $w_1$ with $w_2$. Then by the analysis in Section \ref{sec fiber-class-inv} or \cite[Section 1.2]{MP06}, the fiber class relative invariants of $(W_i, Z_{i,0})$ and $(W_i, Z_{i,0}\sqcup Z_{i,\0})$ are identified once we identify the integration over $Z_i$ of cohomology classes in $H^*(Z_i)$. In fact, this is identified by $\F_Z:H^*(Z_1)\cong H^*(Z_2)$ and $\F_Z(z_1)=z_2$. Hence the fiber class relative Gromov--Witten invariants of {$(W_i, Z_{i,0})$ and} $(W_i, Z_{i,0}\sqcup Z_{i,\0})$ are identified via $\F_Z$.

Now we consider rubber invariants of $(W_i,Z_{i,0}\sqcup Z_{i,\0})$. The boundary strata analysis and rigidification procedure are identified via $\F_Z$ and $\Psi$ automatically from the presentation in \cite[Section 1.5]{MP06}. So we only need to show that we can identify distinguished type II invariants of $(W_i,Z_{i,0}\sqcup Z_{i,\0})$. First of all, via $\F_W$ (determined by $\F_Z$) we could get a one-to-one correspondence between distinguished type II invariants of $(W_i,Z_{i,0}\sqcup Z_{i,\0})$ for $i=1,2$. We use $R$ to denote the invariant for $(W_1,Z_{1,0}\sqcup Z_{1,\0})$ and $\F(R)$ for $(W_2,Z_{2,0}\sqcup Z_{2,\0})$. Via the induction \eqref{eq induction-type-II}, one can see that after assuming that all invariants $R'\stackrel{\circ}{<} R$ are identified, once we fix the $H\in H^*(Z_1)$ for $(W_1,Z_{1,0}\sqcup Z_{1,\0})$ in \eqref{eq induction-type-II} and use the corresponding $\F(H)$ for $(W_2,Z_{2,0}\sqcup Z_{2,\0})$ in \eqref{eq induction-type-II}, the invariant $R$ of $(W_1,Z_{1,0}\sqcup Z_{1,\0})$ is identified with the invariant $\F(R)$ of $(W_2,Z_{2,0}\sqcup Z_{2,\0})$ if the Hodge integrals in twisted Gromov--Witten invariants of $Z_i$ with twisting coming from $\mc O_{Z_i}(-1)$ are identified. However, via the Quantum Riemann--Roch \cite{CG07} the twisted invariants of $Z_i$ are determined by invariants of $Z_i$ and $c_1(\mc O_{Z_i}(-1))$. Since $\F_Z(c_1(\mc O_{Z_1}(-1)))=c_1(\mc O_{Z_2}(-1))$ that is $\F_Z(-z_1)=-z_2$ by the definition of $\F_Z$, the result for absolute Gromov--Witten invariants of projective bundles of
%Fan--Lee \cite[Theorem A]{FL16} and
Fan \cite[Theorem A]{Fan17} applies. This finished the proof of Theorem \ref{thm main}.

\subsection{Gromov--Witten theory of blow-up along complete intersection}
By Theorem \ref{thm main}, we could give an alternative proof of \cite[Conjecture 2]{MP06} besides the proofs given by Fan in \cite{Fan17}, and by Chen, Wang and the author in \cite{DCW17}.

\begin{conj}[Maulik--Pandharipande]
Let $Z\subset X$ be a complete intersection of $l$ divisors $W_1,\ldots,W_l$ of $X$, $\tilde X$ be the blow-up of $X$ along $Z$. The descendent absolute Gromov--Witten invariants of $\tilde X$ are determined by the restriction map $H^*(X)\rto H^*(Z)$ and descendent absolute Gromov--Witten invariants of $X$ and $Z$.
\end{conj}

\begin{proof}
By the degeneration formula (cf. \cite{LR01,Li02,IP04}), Gromov--Witten invariants of $\tilde X$ are determined by relative Gromov--Witten invariants of $(X,\P(N))$, and $(\P(N\oplus\mc O_Z),\P(N))$, where $N$ is the normal bundle of $Z$ in $X$. Maulik--Pandharipande proved that relative Gromov--Witten invariants of $(X,\P(N))$ are determined by absolute invariants of $X$ and $\P(N)$. By Theorem \ref{thm main}, relative Gromov--Witten invariants of $(\P(N\oplus\mc O_Z),\P(N))$ are determined by Gromov--Witten invariants of $Z$ and $c(N)$. Meanwhile $c(N)$ is determined by $H^*(X)\rto H^*(Z)$. On the other hand, absolute Gromov--Witten invariants of $\P(N)$ is determined by Gromov--Witten invariants of $Z$ and $c(N)$ too as proved by Fan \cite{Fan17}. This proves the conjecture.
\end{proof}

Fan \cite{Fan17} proved this conjecture by first reducing relative Gromov--Witten invariants of $(\P(N\oplus\mc O_Z),\P(N))$ to absolute Gromov--Witten invariants of $\P(N\oplus\mc O_Z)$ and $\P(N)$ via the absolute/relative correspondence result of Maulik--Pandharipande \cite{MP06} and Hu--Li--Ruan\cite{HLR08}, then applying results on the uniqueness of absolute Gromov--Witten invariants of projective bundles proved in \cite{FL16, Fan17}.

On the other hand, for the case considered in the conjecture, $N$ splits into a direct sum of line bundles. By Maulik--Pandharipande's result, the relative invariants of $(X,\P(N))$ and $(\P(N\oplus\mc O_Z),\P(N))$ are determined by absolute Gromov--Witten invariants of $X,\P(N),\P(N\oplus\mc O_Z)$. We could then apply virtual localization \cite{GP99} with respect to a $(\cplane^*)^l$-action on $\P(N)$ and $\P(N\oplus\mc O_Z)$ to show that the absolute invariants of $\P(N)$ and $\P(N\oplus\mc O_Z)$ are determined by twisted invariants of $Z$ with twisting coming from $N$, which by Quantum Riemann--Roch \cite{CG07} are determined by Gromov--Witten invariants of $Z$ and $c(N)$. This procedure was done previously by Chen, Wang and the author in \cite{DCW17} for complete intersection in the more general orbifold case.

\subsection*{Acknowledgements}
The author thanks the anonymous referees for valuable suggestions. This work was supported by National Natural Science Foundation of China (Grant No. 11501393).

\end{document}